\documentclass[10pt, article]{amsart}
\usepackage{tikz}
\usetikzlibrary{calc}
\usepackage{ae} 
\usepackage[T1]{fontenc}
\usepackage[cp1250]{inputenc}
\usepackage{amsmath}
\usepackage{amssymb, amsfonts,amscd,verbatim}

\usepackage[normalem]{ulem}
\usepackage{hyperref}
\usepackage{indentfirst}
\usepackage{latexsym}
\input xy
\xyoption{all}

\usepackage{xcolor}

\usepackage{amsmath}    

\theoremstyle{plain}
\newtheorem{Pocz}{Poczatek}[section]
\newtheorem{Proposition}[Pocz]{Proposition}

\newtheorem{Theorem}[Pocz]{Theorem}
\newtheorem{Corollary}[Pocz]{Corollary}

\newtheorem{Lemma}[Pocz]{Lemma}
\newtheorem{Observation}[Pocz]{Observation}

\theoremstyle{definition}
\newtheorem{Definition}[Pocz]{Definition}

\theoremstyle{remark}
\newtheorem{Remark}[Pocz]{Remark}

\numberwithin{equation}{section}
\title[Universal spaces for asymptotic dimension zero]
{Universal spaces for asymptotic dimension zero}

\author{Yuankui Ma}
\address{Xi'an Technological University, No.2 Xuefu zhong lu, Weiyang district, Xi'an, China 710021}
\email{mayuankui@xatu.edu.cn}

\author{Jeremy Siegert}
\address{University of Tennessee, Knoxville, TN 37996, USA}
\email{jsiegert@vols.utk.edu}

\author{Jerzy Dydak}
\address{University of Tennessee, Knoxville, TN 37996, USA}
\email{jdydak@utk.edu}
\address{Xi'an Technological University, No.2 Xuefu zhong lu, Weiyang district, Xi'an, China 710021}
\email{jdydak@gmail.com}

\date{ \today
}


\begin{document}
\maketitle
\begin{center}
\today
\end{center}

\tableofcontents

\begin{abstract}
Dranishnikov and Zarichnyi constructed a universal space in the coarse category of spaces of bounded geometry of asymptotic dimension $0$. In this paper we construct universal spaces in the coarse category of separable (respectively, proper) metric spaces of asymptotic dimension $0$. Our methods provide an alternative proof of Dranishnikov-Zarichnyi result.
\end{abstract}

\section{Introduction}

The asymptotic dimension of metric spaces was introduced by Gromov in \cite{Grom} as a large scale analog of the small scale covering dimension used in traditional topology. The property has been widely studied due in part to its application towards progress on the Novikov conjecture by Yu in \cite{Yu}, but also because of its own geometric appeal in large scale dimension theory. In the study of asymptotic dimension, as with the classical covering dimension, it is of interest to construct universal spaces for particular dimensions. That is, for a particular class of spaces $\mathcal{H}$ all have dimension $n$, with respect to some definition of dimension, we say that a space $X\in\mathcal{H}$ is universal with respect to dimension $n$ if every $Y\in\mathcal{H}$ emebeds into $X$ (where the nature of the embedding may be topological or coarse depending on the context one is working). In the case of proper metric spaces with bounded geometry and asymptotic dimension $n$ this was done by Dranishnikov and Zarichnyi in \cite{DZ}. In this paper we will construct a universal space (with respect to coarse embeddings) for two separate classes of metric spaces of asymptotic dimension $0$. The first is the class of separable metric spaces of asymptotic dimension $0$ and the second is the class of proper metric spaces of asymptotic dimension $0$. This is done by using the observation made in \cite{BDHM} that every separable metric space of asymptotic dimension $0$ is coarsely equivalent to an integral ultrametric space. From there one constructs spaces that are universal (with respect to isometric embeddings) in specific classes of bounded ultrametric spaces. Finally, the aforementioned universal spaces are constructed. 

Notice that T. Banakh and I. Zarichinyy \cite{BZ}  constructed universal spaces for coarsely homogeneous spaces of asymptotic dimension $0$.

\section{Preliminaries}

We begin with the few basic preliminary definitions needed in later sections. 

\begin{Definition}
A metric space $(X,d)$ is called an \textbf{ultrametric space} if in place of the usual triangle inequality for metric spaces the metric $d$ satisfies the stronger \emph{ultrametric triangle inequality} which says that for all $x,y,z\in X$

\[d(x,z)\leq\max\{d(x,y),d(y,z)\}\]
\end{Definition}

\begin{Definition}
Given a set $D$ of non-negative integers, a \textbf{$D$-ultrametric space} is an ultrametric space with all distances belonging to $D$.
If $D$ is the set of all non-negative integers, then a $D$-ultrametric space will be called an \textbf{integral ultrametric space}.
\end{Definition}

\begin{Definition}
A metric space $(X,d)$ is said to be of \text{asymptotic dimension} $0$ if for every uniformly bounded cover $\mathcal{U}$ of $X$, there is a uniformly bounded cover $\mathcal{V}$ that is refined by $\mathcal{U}$ and whose elements are disjoint. 
\end{Definition}

Alternatively, one could define for each $r>0$ the relation $\sim_{r}$ on $X$ be setting $x\sim_{r}y$ if $d(x,y)<r$. Then say that $x,y\in X$ are \textbf{r-connected} if there is a finite chain of elements $x=y_{0},y_{1},\ldots,y_{n}=y$ such that $y_{i}\sim_{r}y_{i+1}$ for $0\leq i\leq n-1$ and define the $\textbf{r-components}$ of $X$ to be maximally $r$-connected subsets of $X$. A metric space $(X,d)$ is of asymptotic dimension $0$ if and only if the collection of $r$-components of $X$ is uniformly bounded for every $r>0$. For a more in depth discussion of asymptotic dimension the reader is referred to \cite{Roe lectures}.

\begin{Definition}
A function $f:(X,d_{1})\rightarrow(Y,d_{2})$ is called:
\begin{enumerate}

\item \textbf{uniformly bornologous} if for all $R>0$ there is an $S>0$ such that if $d_{1}(x,y)\leq R$ then $d_{2}(f(x),f(y))\leq S$.
\item \textbf{proper} if for every bounded $B\subseteq Y$, $f^{-1}(B)$ is bounded in $X$.
\item \textbf{uniformly proper} if for every $R>0$ there is an $S>0$ such that if $B\subseteq Y$ is bounded by $R$, then $f^{-1}(B)$ is bounded by $S$.
\item \textbf{coarsely surjective} if there is an $R>0$ such that for every $y\in Y$ there is an $x\in X$ such that $d_{2}(f(x),y)\leq R$.
\item a \textbf{coarse equivalence} if it is uniformly bornologous, uniformly proper, and coarsely surjective.
\item a \textbf{coarse embedding} if it is uniformly bornologous and uniformly proper.
\end{enumerate}
\end{Definition}


The following result from \cite{BDHM} is the critical observation needed to construct the universal spaces in section \ref{universal spaces}.

\begin{Theorem}\label{countableintegral}
If $(X,d)$ is a separable metric space of asymptotic dimension zero, then
there is a countable integral ultrametric space $(Y,\rho)$ coarsely equivalent to $(X,d)$. Moreover, if $X$ is proper, then $Y$ can be chosen to have finite bounded subsets only.
\end{Theorem}

\section{Coarse disjoint unions}\label{CoarseDisjointUnions}
In this section we define the coarse disjoint union, which is a means of joining a family of metric spaces together in such a way that they are "coarsely independent" of one another. 

\begin{Definition}\label{DefCoarseDisjointUnion}
Given a family $\{(X_s,d_s)\}_{s\in S}$ of metric spaces, a \textbf{coarse disjoint union} of that family is a disjoint union $\coprod\limits_{s\in S} X_s$
equipped with a metric $d$ satisfying the following properties:\\
1. $d$ restricted to each $X_s$ equals $d_s$.\\
2. Given $M > 0$ there are bounded subsets $B_s$ of $X_s$, all but finitely many of them empty, such that if $x\in X_s\setminus B_s$ and $y\in X_t\setminus B_t$ for some $s\ne t$, then $d(x,y) > M$.
\end{Definition}

\begin{Observation} Notice that Condition 2 in \ref{DefCoarseDisjointUnion} can be split into the following two conditions:\\
2a. Every bounded subset $B$ of $\coprod\limits_{s\in S} X_s$ is contained
in $\coprod\limits_{s\in F} X_s$ for some finite $F\subset S$.\\
2b. Given $M > 0$ there is a bounded subset $B$ of $\coprod\limits_{s\in S} X_s$ such that if $x\in X_s\setminus B$ and $y\in X_t\setminus B$ for some $s\ne t$, then $d(x,y) > M$.
\end{Observation}

\begin{Observation}
If a coarse disjoint union exists and each $X_s$ is non-empty, then $S$ is countable.
\end{Observation}

\begin{Proposition}
A disjoint union $\coprod\limits_{s\in S} X_s$ is a coarse disjoint union of a family $\{(X_s,d_s)\}_{s\in S}$ of metric spaces if and only if is
equipped with a metric $d$ satisfying the following properties:\\
1. $d$ restricted to each $X_s$ equals $d_s$.\\
2. Given a sequence $\{x_n\}_{n\ge 1}$ of points in $\coprod\limits_{s\in S} X_s$ belonging to different parts $X_s$, one has $x_n\to\infty$ (that means $d(a,x_n)\to\infty$ for some, hence for all, $a\in \coprod\limits_{s\in S} X_s$).\\
3. Given $M > 0$ and a sequence of pairs $(x_n,y_n)$, $n\ge 1$, of points in $\coprod\limits_{s\in S} X_s$ such that $d(x_n,y_n) < M$ for all $n\ge 1$ and $x_n\to\infty$, there is $k\ge 1$ such that for each $n\ge k$ there is an index $s\in S$ so that $x_n,y_n\in X_s$.
\end{Proposition}
\begin{proof}
Suppose is a coarse disjoint union in the sense of Definition \ref{DefCoarseDisjointUnion}. Given a sequence $\{x_n\}_{n\ge 1}$ of points in $\coprod\limits_{s\in S} X_s$ belonging to different parts $X_s$ such that $d(a,x_n)$ is not divergent to infinity for some $a\in \coprod\limits_{s\in S} X_s$, we may reduce this case to the one where there is $M > 0$ satisfying
$d(a,x_n) < M$ for all $n\ge 1$. There are bounded subsets $B_s$ of $X_s$, all but finitely many of them empty, such that if $x\in X_s\setminus B_s$ and $y\in X_t\setminus B_t$ for some $s\ne t$, then $d(x,y) > 2M$.
There are $t\ne s$ in $S$ such that $B_t=B_s=\emptyset$ and $x_k\in X_t$, $x_m\in X_s$ for some $k,m$, a contradiction as $d(x_k,x_m) < 2M$.

Given $M > 0$ and a sequence of pairs $(x_n,y_n)$, $n\ge 1$, of points in $\coprod\limits_{s\in S} X_s$ such that $d(x_n,y_n) < M$ for all $n\ge 1$ and $x_n\to\infty$, assume there is no $k\ge 1$ such that for each $n\ge k$ there is an index $s\in S$ so that $x_n,y_n\in X_s$. We may reduce this case to the one where $x_n$ and $y_n$ do not belong to the same $X_s$ for all $n\ge 1$. There are bounded subsets $B_s$ of $X_s$, all but finitely many of them empty, such that if $x\in X_s\setminus B_s$ and $y\in X_t\setminus B_t$ for some $s\ne t$, then $d(x,y) > M$.
There are $t\ne s$ in $S$ such that $B_t=B_s=\emptyset$ and $x_k\in X_t$, $y_k\in X_s$ for some $k$, a contradiction as $d(x_k,y_k) < M$.

The proof in the reverse direction is similar.
\end{proof}

\begin{Proposition}\label{TwoDifferentCoarseUnionsProp}
Given two coarse disjoint unions $(\coprod\limits_{s\in S} X_s,d)$ and $(\coprod\limits_{s\in S} Y_s,\rho)$ \\
1. isometric embeddings $i_s:X_s\to Y_s$, $s\in S$, induce a coarse embedding $i$ from $(\coprod\limits_{s\in S} X_s,d)$ to $(\coprod\limits_{s\in S} Y_s,\rho)$,\\
2. identity functions $i_s:X_s\to Y_s$, $s\in S$, induce a coarse equivalence $i$ from $(\coprod\limits_{s\in S} X_s,d)$ to $(\coprod\limits_{s\in S} Y_s,\rho)$.
\end{Proposition}
\begin{proof} Notice 1) implies 2), so only 1) needs to be proved.
Suppose, on the contrary, that there is $M > 0$ and a sequence of pairs $(x_n,y_n)$, $n\ge 1$, of points in $\coprod\limits_{s\in S} X_s$ such that $d(x_n,y_n) < M$ for all $n\ge 1$ but
$\rho(i(x_n),i(y_n))\to\infty$. There is $k\ge 1$ such that for each $n\ge k$ there is an index $s\in S$ so that $i(x_n),i(y_n)\in Y_s$. Therefore $\rho(i(x_n),i(y_n))=d(x_n,y_n)$ for all $n > k$, a contradiction.
\end{proof}

\begin{Lemma}\label{rDisjointUnionLemma}
Suppose $r > 0$ and $(X_i,x_i,d_i)$, $i=1,2$, are two disjoint pointed metric spaces.
The symmetric function $d$ on $(X_1\cup X_2)\times (X_1\cup X_2)$
extending both metrics defined
by $d(x,y)=\max(d_1(x,x_1),r,d_2(y,x_2))$, if $x\in X_1$ and $y\in X_2$, is a metric and $(X_1\cup X_2,d)$ is a coarse disjoint union of $(X_i,d_i)$, $i=1,2$.
Moreover, $d$ is an ultrametric if both $d_i$ are ultrametrics.
\end{Lemma}
\begin{proof}
Suppose $d(x,z) > d(x,y)+d(y,z)$. Therefore all three points cannot belong to one space $X_i$ and it suffices to consider two cases:\\
Case 1. $x,z\in X_1$, $y\in X_2$.\\
Case 2. $x,y\in X_1$, $z\in X_2$.

In Case 1, $ d(x,y)+d(y,z)\ge d_1(x,x_1)+d_1(x_1,z)\ge d_1(x,z)=d(x,z)$, a contradiction.

In Case 2, $d(x,y)+d(y,z)\ge d_1(x,y)+d_1(x_1,y)\ge d_1(x,x_1)$
and $d(x,y)+d(y,z)\ge d(y,z)\ge \max(r,d_2(z,x_2))$,
so finally $d(x,y)+d(y,z)\ge \max(d_1(x,x_1),r,d_2(z,x_2))=d(x,z)$,
a contradiction again.

Assume each $d_i$ is an ultrametric and assume $d(x,z) > \max(d(x,y),d(y,z))$. Therefore all three points cannot belong to one space $X_i$ and it suffices to consider two cases:\\
Case A. $x,z\in X_1$, $y\in X_2$.\\
Case B. $x,y\in X_1$, $z\in X_2$.

In Case A, $\max(d(x,y),d(y,z))\ge \max(d_1(x,x_1),d_1(x_1,z))\ge d_1(x,z)=d(x,z)$, a contradiction.

In Case B, $\max(d(x,y),d(y,z))\ge \max(d_1(x,y),d_1(x_1,y))\ge d_1(x,x_1)$
and $\max(d(x,y),d(y,z))\ge d(y,z)\ge \max(r,d_2(z,x_2))$,
so finally $\max(d(x,y),d(y,z))\ge \max(d_1(x,x_1),r,d_2(z,x_2))=d(x,z)$,
a contradiction again.

If $M > 0$, put $B_1=B(x_1,M+1)$, $B_2=B(x_2,M+1)$ and notice $d(x,y) > M$
if $x\in X_1\setminus B_1$ and $y\in X_2\setminus B_2$.
Thus $(X_1\cup X_2,d)$ is a coarse disjoint union of $(X_i,d_i)$, $i=1,2$.
\end{proof}

\begin{Definition}\label{rDisjointUnionDef}
Suppose $r > 0$ and $(X_i,x_i,d_i)$, $i=1,2$, are two disjoint pointed metric spaces. The space constructed in \ref{rDisjointUnionLemma} will be called
the \textbf{$r$-union} of $(X_i,x_i,d_i)$, $i=1,2$.

Notice that if both $X_i$ are bounded and $r\ge \max(diam(X_1),diam(X_2))$,
then the distance from $x\in X_1$ to $y\in X_2$ is always $r$. That means base points are irrelevant in such a case and we can talk about the \textbf{$r$-union} of $(X_i,d_i)$, $i=1,2$.
\end{Definition}

\begin{Definition}\label{InfiniterDisjointUnionDef}
Suppose $r_i$, $i\ge 1$, is a (possibly finite) sequence of positive numbers and $(X_i,x_i,d_i)$, $i\ge 1$, are mutually disjoint pointed metric spaces. The
the \textbf{$\{r_i\}_{i\ge 1}$-union} of $(X_i,x_i,d_i)$, $i\ge 1$, is 
as the union of spaces $(Y_n,\rho_n)$, $n\ge 1$, defined inductively as follows:\\
1. $(Y_1,\rho_1)=(X_1,d_1)$.\\
2. $Y_{n+1}$ is the $r_n$-union of $(Y_n,x_n,\rho_n)$ and $(X_{n+1},x_{n+1},d_{n+1})$.

Notice that if all $X_i$ are bounded, $\{r_i\}_{i\ge 1}$ is an increasing sequence, and $r_n\ge diam(X_{n+1})$ for each $n\ge 1$,
then the distance from $x\in X_i$ to $y\in X_j$ is always $r_{j-1}$ if $i < j$. That means base points are irrelevant in such a case and we can talk about the \textbf{$\{r_i\}_{i\ge 1}$-union} of $(X_i,d_i)$, $i\ge 1$.
\end{Definition}

\begin{Lemma}\label{InfiniteDisjointUnionIsCoarse}
Suppose $r_i$, $i\ge 1$, is a sequence of positive numbers and $(X_i,x_i,d_i)$, $i\ge 1$, are mutually disjoint pointed metric spaces. The
the $\{r_i\}_{i\ge 1}$-union of $(X_i,x_i,d_i)$, $i\ge 1$, is a coarse disjoint union of $(X_i,d_i)$, $i\ge 1$, if $r_i$, $i\ge 1$, is finite and also if $r_i$ is diverging to infinity.
\end{Lemma}
\begin{proof}
Assume $M > 0$. If $r_i$ is a finite sequence, put $B_i=B(x_i,M+1)$.
If $r_i$ is infinite choose $k$ such that $r_i > M$ for each $i\ge k$ and put $B_n=\emptyset$ for $n > k$. Notice $d(x,y) > M$ if $x\in X_i\setminus B_i$, $y\in X_j\setminus B_j$, and $i\ne j$.
\end{proof}

\begin{Corollary}
If $S$ is countable, then any family $\{(X_s,d_s)\}_{s\in S}$ of metric spaces has a coarse disjoint union. Moreover, if each $d_s$ is an (integral) ultrametric, then there is a coarse disjoint union equipped with an (integral) ultrametric.
\end{Corollary}

\begin{Proposition}\label{SeparabilityPropernessOfUnions}
Suppose $S$ is countable and $\{(X_s,d_s)\}_{s\in S}$ is a family of metric spaces. A coarse disjoint union of $\{(X_s,d_s)\}_{s\in S}$ is separable (proper) if and only if each $X_s$ is separable (proper).
\end{Proposition}
\begin{proof}
It follows from the fact any bounded subset $B$ of the coarse disjoint union is a union of finitely many bounded subsets of some $X_s$.
\end{proof}

\section{Special ultrametric spaces}\label{SpecialUltrametricSpaces}

In this section we construct universal spaces (with respect to isometric embeddings) for specific classes of bounded ultrametric spaces. More specifically, for a finite subset $D\subseteq\mathbb{N}$ that contains $0$ we construct universal spaces for the class of countable $D$-ultrametric spaces, and for each $m\geq 1$ we construct an universal space for the class of $D$-ultrametric spaces with at most $m$ points. The spaces constructed in this section serve as the building blocks for the universal spaces constructed in section \ref{universal spaces}.

\begin{Lemma}\label{BasicUltraSplittingLemma}
Suppose $(X,d)$ is an ultrametric space, $x_0\in X$, $0 < r < s$,
$x_1, x_2\in X$, and $d(x_1,x_0)=r$, $d(x_2,x_0)=s$. If there are no points in $X$ such that $r < d(x,x_0) < s$ or $d(x,x_0) > s$, then
$(X,d)$ is isometric to the $s$-union of $X_1$ and $X_2$,
where $X_1:= \{x\in X | d(x,x_0) \leq r\}$ and $X_2:= \{x\in X |  d(x,x_0) =s\}$.
\end{Lemma}
\begin{proof}
Notice that $diam(X_1)\leq r$ and $diam(X_2)\leq s$. Given $x\in X_1$
and $y\in X_2$ one has $d(x,y)=s$. Indeed, $d(y,x_0)=s$, $d(x,x_0)\leq r < s$, so $d(y,x_0)\leq \max(d(x,y),d(x,x_0))$ is possible only if $d(x,y)=s$.
That is sufficient to conclude the proof.
\end{proof}

\begin{Corollary}
\label{UltraAsDisjUnion}
Suppose $(X,d)$ is an ultrametric space, $x_0\in X$ and points $x_n\in X\setminus \{x_0\}$, $n\ge 1$, are chosen such that
the sequence $\{r_n=d(x_n,x_0)\}_{n\ge 1}$ is strictly increasing and diverges to infinity. Put $X_1= \{x\in X | d(x,x_0) \leq r_{1}\}$ , and define $X_{n}$, $n \ge 2$, as $\{x\in X | d(x,x_0)= r_{n}\}$. If for every $x\in X\setminus X_1$ there is $i\ge 2$
such that $d(x,x_0)=r_i$, then
the natural function from the $\{r_n\}_{n\ge 1}$-union of all $(X_n,d)$ to
$X$ is an isometry. Thus $(X,d)$ is a coarse disjoint union of some of its bounded subsets.
\end{Corollary}
\begin{proof}
Notice spaces $Y_{n+1}$ in \ref{InfiniterDisjointUnionDef} are identical with
$\{x\in X | d(x,x_0) \leq r_{n}\}$ by applying \ref{BasicUltraSplittingLemma}.
\end{proof}

\begin{Lemma}
Suppose $D$ is a finite set of non-negative integers and $(X,d)$ is a $D$-ultrametric space.
Given $m=\max(D) > 0$ 
the relation $x\sim y$ defined as $d(x,y) < m$ is an equivalence relation
such that the distance between points in different equivalence classes is exactly $m$. Therefore $(X,d)$ is isometric to the $\{r_i\}_{i=1}^k$-union of
all the equivalence classes, where $r_i=m$ for each $i\leq k$.
\end{Lemma}
\begin{proof}
The relation is clearly an equivalence one due to the fact $d(x,y)\leq \max(d(x,z),d(y,z))$ for all $x,y,z\in X$. Also, points in different equivalence classes are at distance $m$.
\end{proof}

\begin{Proposition}\label{boundedfiniteuniversal}
Given a finite set $D$ of non-negative integers and given $m\ge 1$ there is a finite ultrametric space $FU(m,D)$ such that any $D$-ultrametric space $X$ containing at most $m$ points isometrically embeds in $FU(m,D)$.
\end{Proposition}
\begin{proof}
Let $FU(m,\{0\})$ be a one-point metric space. Suppose spaces $FU(m,D)$ are known for all $D$ containing at most $n$ integers and $C$ contains $(n+1)$ integers with $k=\max(C)$. Define $FU(m,C)$ as the $\{r_i\}_{i=1}^m$-union of $FU(m,C\setminus \{k\})$ where $r_i=k$ for each $i\leq m$.
\end{proof}

\begin{Proposition}\label{boundedcountableuniversal}
Given a finite set $D$ of non-negative integers there is a countable $D$-ultrametric space $CU(D)$ such that any countable $D$-ultrametric space $X$ isometrically embeds in $CU(D)$.
\end{Proposition}
\begin{proof}
Let $CU(\{0\})$ be a one-point metric space. Suppose spaces $CU(D)$ are known for all $D$ containing at most $n$ integers and $G$ contains $(n+1)$ integers with $k=\max(G)$. Define $CU(G)$ as the $\{r_i\}_{i=1}^m$-union of $CU(G\setminus \{k\})$ where $r_i=k$ for each $i\leq m$.
\end{proof}

\section{Universal spaces}\label{universal spaces}

In this last section we prove our main results. That is, we give a detailed construction of universal spaces (with respect to coarse embeddings) in the classes of separable metric spaces of asymptotic dimension $0$ and the class of proper metric spaces of asymptotic dimension $0$. For the following two results, let $D_{1},D_{2},\ldots$ be an enumeration of the finite subsets of $\mathbb{N}$ that contain $0$. For each $i\ge 1$ put $r_i=\max(D_i)$
and notice $r_i\to\infty$.

\begin{Theorem}
There is a countable integral ultrametric space $CU$ such that any separable metric space $X$ of asymptotic dimension $0$ coarsely embeds in $CU$.
\end{Theorem}
\begin{proof}
Define $CU$ as the $\{r_i\}_{i\ge 1}$-union of all $CU(D_{i})$. We claim that $CU$ is the desired universal space. In light of Theorem \ref{countableintegral} it will suffice to show that if $(X,d)$ is a countable integral ultrametric space, then $X$ coarsely embeds into $CU$. Then let $(X,d)$ be such a space. By Corollary \ref{UltraAsDisjUnion} $X$ can be written as a coarse disjoint union of bounded $G_k$-ultrametric spaces $X_k$. There is a strictly increasing sequence $(n_{k})_{k\in\mathbb{N}}$ such that $G_k\subset D_{n_{k}}$ for each $k\ge 1$. Then, by Proposition \ref{TwoDifferentCoarseUnionsProp}, we have that $X$ embeds into $CU$.
\end{proof}

yyy

\begin{Theorem}
There is a countable and proper integral ultrametric space $PU$ such that any proper metric space $X$ of asymptotic dimension $0$ coarsely embeds in $PU$.
\end{Theorem}
\begin{proof}
We again use the enumeration $D_{1},D_{2},\ldots$ of the finite subsets of $\mathbb{N}$ that contain $0$. The set of all $FU(m,D_{n})$ (where defined) is countable. Enumerate these spaces and denote this sequence $\{Y_{1},Y_{2},\ldots\}$. Let $r_i=i+\sum\limits_{j=1}^i diam(Y_j)$ for $i\ge 1$.

We then define $PU$ is the $\{r_i\}$-union of all $Y_{i}$. It is proper by \ref{SeparabilityPropernessOfUnions}. Let $(X,d)$ be a countable proper metric space of asymptotic dimension $0$. By Theorems \ref{countableintegral} and \ref{UltraAsDisjUnion} we may assume without loss of generality that that $X$ can be written as a coarse disjoint union of finite $G_k$-ultrametric spaces $X_k$. There is a strictly increasing sequence $(n_{k})_{k\in\mathbb{N}}$ such that $G_k\subset D_{n_{k}}$ for each $k\ge 1$ and $D_{n_{k}}$ is a proper subset of $D_{n_{k+1}}$ for each $k\ge 1$.
Then, by Proposition \ref{TwoDifferentCoarseUnionsProp} we have that $X$ coarsely embeds into $PU$.
\end{proof}

\section{Ultrametric groups as universal spaces}

In this section we show that certain unbounded ultrametric groups are universal in respective categories of spaces of asymptotic dimension $0$.

\begin{Definition}
An \textbf{(integral) ultrametric group} is a group equipped with a left-invariant (integral) ultrametric $d$. 
\end{Definition}

\begin{Proposition}\label{UltrametricsInducedBySubgroups}
Suppose $G$ is a group and $D$ is a discrete subset of non-negative reals containing $0$. Assigning $G$ a left-invariant $D$-ultrametric $d$ is equivalent to picking subgroups $G_a$, $a\in D$, of $G$ satisfying the following conditions:\\
1. $G_0=\{1_G\}$,\\
2. $G_a$ is a subgroup of $G_b$ if $a < b$ belong to $D$,\\
3. $\bigcup\limits_{a\in D}G_a=G$.
\end{Proposition}
\begin{proof}
Given a left-invariant $D$-ultrametric $d$ on $G$ and given $a\in D$ define $G_a$ as all $g\in G$ satisfying $d(g,1_G)\leq a$. Notice $g\in G_a$ implies $g^{-1}\in G_a$ as $d(g^{-1},1_G)=d(g\cdot g^{-1},g\cdot 1_G)=d(1_G,g)$.
Also, if $g,h\in G_a$, then $d(g\cdot h,1_G)\leq \max(d(g\cdot h,g),d(g,1_G))=\max(d(h,1_G),d(g,1_G))\leq a$. It is obvious that $\{G_a\}_{a\in D}$ satisfy Conditions 1-3.

Given $\{G_a\}_{a\in D}$ satisfying Conditions 1-3 define $d(g,h)$
as the infimum of $a\in D$ satisfying $g^{-1}\cdot h\in G_a$. If $d(g,h),d(h,k)\leq a$, then $g^{-1}\cdot h\in G_a$ and $ h^{-1}\cdot k\in G_a$,
so their product $g^{-1}\cdot k$ belongs to $G_a$ and $d(g,k)\leq a$. That means $d$ is an ultrametric, indeed.
\end{proof}

\begin{Definition}
Given a discrete subset $D$ of non-negative reals containing $0$ and given subgroups $G_a$, $a\in D$, of $G$ satisfying the following conditions:\\
1. $G_0=\{1_G\}$,\\
2. $G_a$ is a subgroup of $G_b$ if $a < b$ belong to $D$,\\
3. $\bigcup\limits_{a\in D}G_a=G$,\\
the ultrametric $d$ in \ref{UltrametricsInducedBySubgroups} is said to be \textbf{induced} by $\{G_a\}_{a\in D}$.
\end{Definition}

\begin{Proposition}\label{CoarseEquivalenceUltrametricStructures}
Suppose $G$ is a group and $d_i$, $i=1,2$, are two ultrametric metrics induced by families $\{G^i_a\}_{a\in D_i}$ of subgroups of $G$.
$(G,d_1)$ is coarsely equivalent to $(G,d_2)$ if and only if for each $a\in D_i$ there is $b\in D_j$, $j\ne i$, such that $G^i_a\subset G^j_b$.
\end{Proposition}
\begin{proof}
Assume the identity $(G,d_1)\to (G,d_2)$ is large scale continuous (aka bornologous)
and $a\in D_1$. There is $b\in D_2$ such that $d_1(g,h)\leq a$ implies $d_2(g,h)\leq b$, so $g\in G^1_a$ implies $g\in G^1_b$
as $d_1(g,1_G)\leq a$ implies $d_2(g,1_G)\leq b$ and $g\in G^2_b$.

The reverse implication is similar.
\end{proof}

\begin{Proposition}\label{CoarseEmbeddingsViaBoundedSubsets}
Suppose $(X,d_X)$ is an integral ultrametric space and $(G,d_G)$ is an integral ultrametric group. If every bounded subset $B$ of $X$ isometrically embeds in $G$, then $(X,d_X)$ coarsely embeds in $(G,d_G)$.
\end{Proposition}
\begin{proof}
Of interest is only the case of $X$ being unbounded, so $G$ is also unbounded. Pick $x_0\in X$ and a sequence $\{x_n\}_{n\ge 1}$ of points in $X$ such that $d(x_{n+1},x_0) > d(x_n,x_0)+1$ for each $n\ge 1$.
Put $r_n=d(x_n,x_0$ for $n\ge 1$ and pick an isometric embedding
$i_n:B_n\to G$, where $B_n=\{x\in X | r_{n-1} < d(x,x_0) \leq r_n$ for $n\ge 2$ and $B_1=\{x\in X | d(x,x_0) \leq r_1$. We may assume $i_n(x_n)=1_G$ for each $n\ge 1$.
Now pick a sequence $\{g_n\}_{n\ge 1}$ of elements of $G$ such that $d_G(g_1,1_G) > r_1$
and $s_n:=d_G(g_n,1_G) > d(g_{n-1},1_G)+diam(B_n)$. Replacing $i_n$
by $j_n:=g_n\cdot i_n$ we obtain a sequence of isometric embeddings
of $B_n$ into $C_n:= \{g\in G | s_{n-1} < d(g,1_G) \leq s_n\}$.
By \ref{UltraAsDisjUnion} $X$ coarsely embeds in $G$.
\end{proof}

\begin{Corollary}\label{BasicCoarseEmbeddingIntoAGroup}
Suppose $(X,d_X)$ is an integral ultrametric space such that for each $n$ 
there is a cardinal number $c(n)$ with the property that each ball $B(x,n+2)$, $x\in X$, has cardinality at most $c(n)$. If $(G,d_G)$ is an integral ultrametric group induced by a sequence of subgroups $\{G_n\}_{n\ge 1}$ with the property that the cardinality of cosets of $G_n$ in $G_{n+1}$ is at least $c(n)$ for each $n\ge 1$, then $(X,d_X)$ coarsely embeds in $(G,d_G)$.
\end{Corollary}
\begin{proof}
Suppose each bounded subset of $X$ of diameter at most $n$ isometrically embeds in $G$. Therefore, for each $x\in X$, there is an isometric embedding
$i_x:B(x,n+1)\to G_n$ such that $i_x(x)=1_G$. Suppose $x_0\in X$. Consider the equivalence relation $x\sim y$ on $B(x_0,n+2)$ defined by $d_X(x,y) < n+1$. For each equivalence class $c$ not containing $x_0$ choose $x(c)\in B(x_0,n+2)\setminus B(x_0,n+1)$ and $g_c\in G_{n+1}$ such that if $c\ne k$, then $g_c^{-1}\cdot g_k\notin G_n$. 
Extend $i_{x_0}$ over $B(x_0,n+2)$ to a function $j$ by sending $x(c)$ to $g_c$ and by sending
any $x$ equivalent to $x_c$ to $g_c\cdot i_{x(c)}(x)$.
Notice $j$ is an isometric embedding when restricted to each equivalence class, the images of different equivalence classes are disjoint, and
if $d_X(x,y)=n+1$, then $d_G(j(x),j(y))=n+1$. That means $j$ is an isometric embedding.
\end{proof}

\begin{Corollary}
\label{UniversalGroupForSeparableCase}
Suppose $G$ is a countable group that is the union of an increasing sequence of its subgroups $\{G_i\}_{i\ge 1}^\infty$ with the property that the index of $G_i$ in $G_{i+1}$ is infinite for each $i\ge 1$. There is an integral ultrametric $d_G$ on $G$ such that $(G,d_G)$ is a universal space in the category of separable metric spaces of asymptotic dimension $0$.
\end{Corollary}

\begin{Corollary}
Let $G$ be a countable vector space over the rationals $Q$ that is of infinite algebraic dimension. There is an integral ultrametric $d_G$ on $G$ such that $(G,d_G)$ is a universal space in the category of separable metric spaces of asymptotic dimension $0$.
\end{Corollary}

\begin{Theorem}
Suppose $G$ is a countable group that is the union of a strictly increasing sequence of its finite subgroups $\{G_i\}_{i\ge 1}^\infty$.There is a proper integral ultrametric $d_G$ on $G$ such that $(G,d_G)$ is a universal space in the category of metric spaces of bounded geometry that have asymptotic dimension $0$.
\end{Theorem}
\begin{proof}
Consider a proper integral ultrametric space $(X,d_X)$ of bounded geometry and choose natural numbers $c(n)$ with the property that each ball $B(x,n+2)$, $x\in X$, contains at most $c(n)$ elements.
Replace $\{G_n\}$ by its subsequence $\{H_n\}$ such that the index of $H_n$ in $H_{n+1}$ is larger than $c(n+1)$
for each $n\ge 1$. 
By \ref{BasicCoarseEmbeddingIntoAGroup} and \ref{CoarseEquivalenceUltrametricStructures}, $(X,d_X)$ coarsely embeds into $(G,d_G)$.
\end{proof}

\begin{Corollary}
Let $G$ be a countable vector space over the $\mathbb{Z}/2\mathbb{Z}$ that is of infinite algebraic dimension. There is a proper integral ultrametric $d_G$ on $G$ such that $(G,d_G)$ is a universal space in the category of metric spaces of bounded geometry that have asymptotic dimension $0$.
\end{Corollary}

\begin{Remark}
See \cite{BZ} and \cite{BHZ} for coarse classifications of groups of asymptotic dimension $0$.
\end{Remark}

\end{document}